\documentclass[11pt,a4paper,twoside,bibliography=totoc]{scrartcl}
\usepackage{a4wide}
\usepackage{amsfonts}
\usepackage{amsmath}
\usepackage{amssymb}
\usepackage{array}
\usepackage{amsthm}
\usepackage{subfigure}
\usepackage{subfig}
\usepackage{mathdots}

\usepackage{algorithm}%,algorithmic}
\usepackage{algorithmic}

\usepackage{graphicx}
\usepackage{color}
\usepackage{pgf}%,pgfarrows,pgfnodes} % pgfarrows, pgfnodes obsolete
\usepackage{scrpage2}
\usepackage{multirow}
\usepackage{listings} %Quellcode einbinden
\usepackage{tikz}
\usepackage{graphicx}
		
\newcommand{\N}{\ensuremath{\mathbb{N}}}
\newcommand{\T}{\ensuremath{\mathbb{T}}}

\newcommand{\NZ}{\ensuremath{\mathbb{N}_{0}}}
\newcommand{\Z}{\ensuremath{\mathbb{Z}}}

\newcommand{\R}{\ensuremath{\mathbb{R}}}
\newcommand{\C}{\ensuremath{\mathbb{C}}}

\newcommand{\ii}{\textnormal{i}}
\newcommand{\e}{\textnormal{e}}

\newcommand{\eip}[1]{\textnormal{e}^{2\pi\ii{#1}}}

\newcommand{\norm}[1]{\left\Vert #1\right\Vert}

\newcommand{\pmat}[1]{\begin{pmatrix} #1 \end{pmatrix}}
\newcommand{\set}[1]{\left\{ #1 \right\}}
\newcommand{\abs}[1]{\left| #1 \right |}

\DeclareMathOperator*{\diag}{diag}

\DeclareMathOperator*{\maxdeg}{maxdeg}
\DeclareMathOperator*{\sep}{sep}

\newcommand{\cond}{\mathrm{cond}}

\newtheorem{thm}{Theorem}[section]
\newtheorem{lemma}[thm]{Lemma}
\newtheorem{remark}[thm]{Remark}
\newtheorem{definition}[thm]{Definition}
\newtheorem{example}[thm]{Example}
\newtheorem{corollary}[thm]{Corollary}
\newtheorem{proposition}[thm]{Proposition}

\numberwithin{equation}{section}
\numberwithin{table}{section}
\numberwithin{figure}{section}

\newcommand{\bend}{\hspace*{0ex} \hfill \hbox{\vrule height
    1.5ex\vbox{\hrule width 1.4ex \vskip 1.4ex\hrule  width 1.4ex}\vrule
    height 1.5ex}}

\long\def\symbolfootnote[#1]#2{\begingroup%
\def\thefootnote{\fnsymbol{footnote}}\footnote[#1]{#2}\endgroup}
\newcommand{\calA}{\mathcal{A}} % "\AA" already defined

\newcommand{\dd}{\mathrm{d}}

\newcommand{\OO}[1]{\mathcal{O}\left(#1\right)}

\renewcommand{\mathbf}[1]{\ensuremath{\boldsymbol{#1}}}

\newcommand{\rank}{ \operatorname{rank}}
\newcommand{\sspan}{\operatorname{span}}

\renewcommand{\thefootnote}{\fnsymbol{footnote}}

\allowdisplaybreaks
\title{A multivariate generalization of Prony's~method}

\date{\today}
\date{}

\author{Stefan Kunis\footnotemark[2]\space\footnotemark[3]
\and Thomas Peter\footnotemark[2]
\and Tim R\"omer\footnotemark[2]
\and Ulrich von der Ohe\footnotemark[2]}

\begin{document}

\maketitle

\begin{abstract}

Prony's method is a prototypical eigenvalue analysis based method
for the reconstruction of a finitely
supported complex measure
on the unit circle
from its moments up to a certain degree.
In this note, we give a generalization of this method to the
multivariate case and prove simple conditions
under which the problem admits a unique solution. % TODO: correct formulation /U
Provided the order of the moments is bounded from below by the number of points
on which the measure is supported
as well as by a small constant divided by the separation distance of these points,
stable reconstruction is guaranteed.
In its simplest form, the reconstruction method consists
of setting up a certain multilevel Toeplitz matrix of the moments,
compute a basis of its kernel,
and
compute by some method of choice the set of common roots of the multivariate polynomials
whose coefficients are given in the second step.
All theoretical results are illustrated by numerical experiments.

\medskip

\medskip

\end{abstract}

\footnotetext[2]{
  Osnabr\"uck University, Institute of Mathematics
  \texttt{\{skunis,petert,troemer,uvonderohe\}@uos.de}
}

\footnotetext[3]{
  Helmholtz Zentrum M\"unchen, Institute of Computational Biology
}

%%%%%%%%%%%%%%%%%%%%%%%%%%%%%%%%%%%%%%%%%%%%%%%%%%%%%%%%%%%%%%%%%%%%%%%%%%%%%%
\section{Introduction}
\label{sect:Einleitung}
In this paper we propose a generalization of de~Prony's classical method~\cite{Pr95}
for the parameter and coefficient reconstruction of univariate finitely supported complex measures
to a finite number of variables.
The method of de~Prony lies at the core of seemingly different classes of problems in signal processing
such as spectral estimation, search for an annihilating filter, deconvolution, spectral extrapolation,
and moment problems.
Thus we provide a new tool to analyze multivariate versions
of a broad set of problems.

To recall the machinery of the classical Prony method
let $\C_\ast:=\C\setminus\{0\}$
and let $\hat{f}_j\in\C_\ast$
and pairwise distinct $z_j\in\C_\ast$, $j=1,\dots,M$, be given.
Let $\delta_{z_j}$ denote the Dirac measure in $z_j$ on $\C_\ast$
and let
\begin{equation*}
  \mu=\sum_{j=1}^M\hat{f}_j\delta_{z_j}
\end{equation*}
be a finitely supported complex measure on $\C_\ast$.
By the Prony method the
$\hat{f}_j$ and $z_j$
are reconstructed from~$2M+1$ moments
\begin{equation*}
  f(k)=\int_{\C_\ast}x^k\dd\mu(x)
  =\sum_{j=1}^M\hat{f}_jz_j^k,
  \quad
  k=-M,\dots,M.
\end{equation*}
Since the coefficients $\hat{p}_\ell\in\C$, $\ell=0,\dots,M$,
of the (not a~priori known) so-called Prony polynomial
\begin{equation*}
  p(Z):=\prod_{j=1}^M(Z-z_j)
  =\sum_{\ell = 0}^M\hat{p}_{\ell}Z^\ell
\end{equation*}
fulfill the linear equations
\begin{equation*}
  \sum_{\ell = 0}^M\hat{p}_\ell f(\ell-m)
  =\sum_{j=1}^M\hat{f}_jz_j^{-m}\sum_{\ell=0}^M\hat{p}_\ell z_j^\ell
  =\sum_{j=1}^M\hat{f}_jz_j^{-m}p(z_j)
  =0,
  \quad
  m=0,\dots,M,
\end{equation*}
and are in fact the unique solution to this system with $\hat{p}_M=1$,
reconstruction of the $z_j$ is possible by computing the kernel vector
$(\hat{p}_1,\dots,\hat{p}_{M-1},1)$ of the rank-$M$ Toeplitz matrix
\begin{equation*}
  T:=(f(k-\ell))_{\substack{\ell=0,\dots,M\\k=0,\dots,M}}\in\C^{M+1\times M+1}
\end{equation*}
and, knowing this to be the coefficient vector of~$p$, compute the roots $z_j$ of $p$.
Afterwards, the coefficients $\hat{f}_j$ of $f$ (that did not enter the discussion until now)
can be uniquely recovered by solving a Vandermonde linear system of equations.
When attempting to generalize this method to finitely supported complex measures on $\C_\ast^d$,
it seems natural to think that the unknown parameters $z_j\in\C_\ast^d$
could be realized as roots of $d$-variate polynomials,
and this is the approach we will follow here.
As in the univariate case,
the coefficients will be given as solution to a suitably constructed
system of linear equations.
However, for $d\ge2$, an added difficulty lies in the fact that
a non-constant polynomial always has uncountably many complex roots,
so that a single polynomial cannot be sufficient to identify the parameters as its roots.
A natural way to overcome this problem is to consider the common roots
of a (finite) set of polynomials.
These sets, commonly called algebraic varieties, are the subject of classical algebraic geometry
and thus there is an immense body of algebraic literature on this topic from which we need only some
basic notions as provided at the end of Section~\ref{sect:pre}.

Our main results are presented in Section~\ref{sect:main}, which is divided into three parts.
In Section~\ref{subsect:algebraic} we prove sufficient conditions to guarantee parameter reconstruction
for multivariate exponential sums by constructing a set of multivariate polynomials such that their common roots are precisely the parameters.
In Section~\ref{subsect:trigonometric} we focus on the case
that the parameters $z_j$ lie on the $d$-dimensional torus,
which allows us to prove numerical stability, provide some implications on the parameter distribution, and construct a single trigonometric polynomial
localized at the parameters. Finally, we state a prototypical algorithm of the multivariate Prony method.

In Section~\ref{sect:ErstesKapitel} we discuss previous approaches towards the multivariate moment problem
as can be found in~\cite{JiSiBe01,AnCaHo10} for generic situations, in~\cite{PoTa132,PlWi13,DiIs15} based on projections of the measure,
and in~\cite{CaFe14,CaFe13,BeDeFe15a,BeDeFe15b} based on semidefinite optimization.
Finally, numerical examples are presented in Section~\ref{sect:num}
and we conclude the paper with a summary in Section~\ref{sect:sum}.

%%%%%%%%%%%%%%%%%%%%%%%%%%%%%%%%%%%%%%%%%%%%%%%%%%%%%%%%%%%%%%%%%%%%%%%%%%%%%%
\section{Preliminaries}\label{sect:pre}
Throughout the paper, the letter $d\in\N$ always denotes the dimension, $\C_\ast:=\C\setminus\set{0}$, and
we let
\begin{equation*}
\C_\ast^d:=(\C_\ast)^d=\{z\in\C^d:\text{$z_\ell\ne0$ for all $\ell=1,\dots,d$}\}
\end{equation*}
be the domain for our parameters. For $z\in\C_\ast^d$, $k\in\Z^d$, we use the multi-index notation
$z^k:=z_1^{k_1}\dotsm z_d^{k_d}$.
We also let $\T:=\{z\in\C:\abs{z}=1\}$ and the $d$-fold Cartesian product $\T^d$
is called \emph{$d$-dimensional torus}.
We start by defining the object of our interest,
that is, multivariate exponential sums, as a natural generalization of univariate exponential sums.
\begin{definition}\label{def:mvexpsum}
A function $f\colon\Z^d\to\C$
is a \emph{$d$-variate exponential sum}
if
there is a finitely supported complex measure $\mu$ on $\C_\ast^d$,
such that for all $k\in\Z^d$, $f(k)$ is the $k$-th moment of $\mu$,
% for all $k\in\Z^d$,
% $f(k)$ is the $k$-th moment of some finitely supported complex measure $\mu$ on $\C_\ast^d$,
that is,
if there are $M\in\N$,
$\hat{f}_1,\dots,\hat{f}_M\in\C_\ast$,
and pairwise distinct $z_1,\dots,z_M\in\C_\ast^d$
such that with $\mu:=\sum_{j=1}^M\hat{f}_j\delta_{z_j}$
we have
\begin{equation*}
  f(k)=\int_{\C_\ast^d}x^k\dd\mu(x)
  =\sum_{j=1}^M\hat{f}_jz_j^k
\end{equation*}
for all $k\in\Z^d$.

In that case $M$, $\hat{f}_j$, and $z_j$, $j=1,\dots,M$,
are uniquely determined,
and $f$ is called~\emph{$M$-sparse},
the $\hat{f}_j$ are called \emph{coefficients} of $f$,
and $z_j$ are called \emph{parameters} of $f$.
The set of parameters of $f$ is denoted by $\Omega_f$
or, if there is no danger of confusion, simply by $\Omega$.
\end{definition}
\begin{remark}\label{rem:trigMoments}
Let $\hat{f}_j\in\C_\ast$
and pairwise distinct $t_j\in[0,1)^d$, $j=1,\dots,M$,
be given.
Then the trigonometric moment sequence of
$\tau=\sum_{j=1}^M\hat{f}_j\delta_{t_j}$,
\begin{equation*}
f\colon\Z^d\to\C,
\quad
k\mapsto\int_{[0,1)^d}\eip{kt}\dd\tau(t)
=\sum_{j=1}^M\hat{f}_j\eip{kt_j},
\end{equation*}
(where $kt$ denotes the scalar product of $k$ and $t$)
is a $d$-variate exponential sum
with parameters $\eip{t_j}=(\eip{t_{j,1}},\dots,\eip{t_{j,d}})\in\T^d$.
This case will be analyzed in detail in Section~\ref{subsect:trigonometric}.
\end{remark}
Let $f\colon\Z^d\to\C$ be an $M$-sparse $d$-variate exponential sum
with coefficients $\hat{f}_j\in\C_\ast$
and parameters $z_j\in\C_\ast^d$, $j=1,\dots,M$.
Our objective is to reconstruct the coefficients and parameters of $f$
given an upper bound $n$ for $M$
and a finite set of samples of $f$
at a subset of $\Z^d$
that depends only on $n$, see also~\cite{PePlSc15}.

The following notations will be used throughout the paper.
For $n\in\N$, let $I_n:=\{0,\dots,n\}^d$
and let $N:=\lvert I_n \rvert=(n+1)^d$.
The multilevel Toeplitz matrix
\begin{equation*}
  T_n(f):=\left(f(k-\ell)\right)_{k,\ell\in I_n}\in\C^{N\times N},
\end{equation*}
which we also refer to as $T_n$,
will play a crucial role in the multivariate Prony method.
Note that the entries of $T_n$ are sampling values of $f$
at a grid of $\lvert I_n-I_n\rvert=(2n+1)^d$ points.

Next we establish the crucial link
between the matrix $T_n$
and the roots of multivariate polynomials.
To this end,
let
\begin{equation*}
  \Pi:=\C[Z_1,\dots,Z_d]
  =\{\sum_{k\in F}p_kZ_1^{k_1}\dotsm Z_d^{k_d}:\text{$F\subset\NZ^d$ finite, $p_k\in\C$}\}.
\end{equation*}
denote the $\C$-algebra of $d$-variate polynomials and for
$p=\sum_kp_kZ_1^{k_1}\dotsm Z_d^{k_d}\in\Pi\setminus\{0\}$ let
\begin{equation*}
  \maxdeg(p):=\max\{\lVert k\rVert_\infty:p_k\ne0\}.
\end{equation*}
The $N$-dimensional subvector space
of $d$-variate polynomials of max-degree at most $n$ is denoted by
\begin{equation*}
  \Pi_n:=\{p\in\Pi\setminus\{0\}:\maxdeg(p)\le n\}\cup\{0\}
  \cong\sspan\{\C^d\ni z\mapsto z^k:k\in I_n\},
\end{equation*}
and the \emph{evaluation homomorphism} at $\Omega=\{z_1,\dots,z_M\}$
will be denoted by
\begin{equation*}
  \calA^\Omega_n\colon\Pi_n\to\C^M,
  \quad
  p\mapsto(p(z_1),\dots,p(z_M)),
\end{equation*}
or simply by $\calA_n$.
Note that the representation matrix of $\calA_n$
w.r.t.~the canonical basis of $\C^M$
and the monomial basis of $\Pi_n$
is given by the multivariate Vandermonde matrix
\begin{equation*}
  A_n=\big(z_j^k\big)_{\substack{j=1,\dots,M\\k\in I_n}}\in\C^{M\times N}.
\end{equation*}
The connection between the matrix $T_n$
and polynomials that vanish on $\Omega$
lies in the observation that,
using Definition~\ref{def:mvexpsum},
the matrix $T_n$ admits the factorization
\begin{equation}\label{eq:factorT}
  T_n=(f(k-\ell))_{k,\ell\in I_n}=P_nA_n^{\top}D_nA_n,
% T_n:=\left(f(k-\ell)\right)_{k,\ell\in I_n}%,\quad % TODO: Notation /U
% =A_n^*\diag(\hat{f})A_n\in\C^{N\times N}
% H:=\left(f(k+\ell)\right)_{k,\ell\in I_n}\quad \in\C^{N\times N}
\end{equation}
with
$D_n=\diag(d)$,
$d_j=z_j^{-n}\hat{f}_j$,
$j=1,\dots,M$,
and a permutation matrix
$P_n\in\{0,1\}^{N\times N}$.
% \todo{$P_n=(e_N,\dots,e_1)$, right?~/U}
Therefore the kernel of $A_n$,
corresponding to the polynomials in $\Pi_n$
that vanish on $\Omega$,
is a subset of the kernel of $T_n$.

In order to deal with the multivariate polynomials
encountered in this way
we need some additional notation.
The \emph{zero locus}
of a set $P\subset\Pi$ of polynomials
is denoted by
\begin{equation*}
V(P):=\{z\in\C^d:\text{$p(z)=0$ for all $p\in P$}\},
\end{equation*}
that is,
$V(P)$ consists of the \emph{common} roots of all the polynomials in $P$.
% Furthermore,
% for the proof of Theorem~\ref{thm:kerA}
% given below and the remainder of this subsection
For a set $\Omega\subset\C^d$,
\begin{equation*}
  I(\Omega):=\{p\in\Pi:\text{$p(z)=0$ for all $z\in\Omega$}\}
  =\bigcup_{n\in\N}\ker\calA^\Omega_n
\end{equation*}
is the so-called \emph{vanishing ideal} of $\Omega$.
Finally,
for a set $P\subset\Pi$ of polynomials,
\begin{equation*}
  \langle P\rangle:=\{\sum_{j=1}^mq_jp_j:\text{$m\in\N$, $q_j\in\Pi$, $p_j\in P$}\}
\end{equation*}
is the \emph{ideal generated by $P$}.
Note that $V(P)=V(\langle P\rangle)$ always holds.
Subsequently,
we identify $\Pi_n$ and $\C^N$
and switch back and forth between
the matrix-vector and polynomial notation.
In particular,
we do not necessarily distinguish between
$\calA_n$ and its representation matrix $A_n$,
so that e.g.~``$V(\ker A_n)$'' makes sense.

\section{Main results}\label{sect:main}
In the following two subsections, we study conditions on the degree $n$, and thereby on the number $(2n+1)^d$ of samples,
such that the parameters $z_j$ can be uniquely recovered and the polynomials used to identify them can be computed
in a numerically stable way.

\subsection{Complex parameters, polynomials, and unique solution}\label{subsect:algebraic}

Our first result gives a simple but nonetheless sharp condition
on the order of the moments such that the set of parameters $\Omega$
and the zero loci $V(\ker A_n)$ and $V(\ker T_n)$ are equal.
\begin{thm}\label{thm:kerA}
Let $f\colon\Z^d\to\C$ be an $M$-sparse $d$-variate exponential sum
with parameters $z_j\in\C_\ast^d$, $j=1,\dots,M$.
If $n\ge M$ then
\begin{equation*}
\Omega_f=V(\ker T_n(f)).
\end{equation*}
Moreover, if this equality holds for all $M$-sparse $d$-variate exponential sums $f$, then $n\ge M$.
\end{thm}
\begin{proof}
Let $\Omega:=\Omega_f=\{z_1,\dots,z_M\}$.
We start by proving $\Omega=V(\ker\calA_n)$.
Since $V(\ker\calA_n)\supset V(\calA_{n+1})\supset\Omega$,
it is sufficient to prove the case $n=M$.
It is a simple fact that
$I(\{z_j\})=\langle Z_1-z_{j,1},\hdots,Z_d-z_{j,d}\rangle$,
and that these ideals are pairwise comaximal, and hence we have
\begin{equation*}
I(\Omega)
=\prod_{j=1}^M I(\{z_j\})
=\langle\prod_{j=1}^M(Z_{\ell_j}-z_{j,\ell_j}):\ell_j\in\{1,\dots,d\}\rangle
\subset\langle\ker\calA_M\rangle\subset I(\Omega),
\end{equation*}
which implies
$\langle\ker\calA_M\rangle=I(\Omega)$.
Thus we have
$V(\ker\calA_M)=V(\langle\ker\calA_M\rangle)=V(I(\Omega))=\Omega$
where the last equality holds because $\Omega$ is finite
(and can easily be derived from the above).

Thus it remains to show that $\ker A_M=\ker T_M$.
We proceed by proving $\rank A_M=M$.
To simplify notation,
we omit the subscript $M$ on the matrices.
Let $N:=\dim\Pi_M$
and suppose that $A\in\C^{M\times N}$ has rank $r<M$.
Let $\Omega^\prime=\{z_1,\dots,z_r\}$
and w.l.o.g.~let $A^\prime\in \C^{r\times N}$,
denoting the first $r$
rows of $A$, be of rank $r$.
Now the first part of the proof implies the contradiction
$\Omega=V(\ker A)=V(\ker A^\prime)=\Omega^\prime$.

Considering the factorization
$T=PA^\top DA$
as in Equation~\eqref{eq:factorT}
and applying Frobenius' rank inequality
(see e.g.~\cite[0.4.5~(e)]{HoJo13})
yields
\begin{equation*}
  \rank A^\top D+\rank DA-\rank D\le\rank A^\top DA=\rank T\le\rank A
\end{equation*}
which implies $\rank T=\rank A=M$.
The factorization clearly implies $\ker A\subset \ker T$ which together with the
rank-nullity theorem $\dim\ker A=N-M=\dim\ker T$ yields the final result.

The converse follows from the fact
that for
$\Omega:=\{(x_j,1,\hdots,1)\in\C_\ast^d:j=1,\hdots,M\}$
with distinct $x_j\in\C_\ast$,
any subset $B\subset\Pi_n$ such that $V(B)=\Omega$
(which holds, by assumption, for $B=\ker T_n$)
necessarily contains a polynomial of (max-)degree at least $M$.
\end{proof}

\begin{example}\label{ex:3points}
 Let $f$ be a $3$-sparse $2$-variate exponential sum with parameters $z_j\in\C_\ast^2$
 and $\Omega=\{z_1,z_2,z_3\}$.
 The generating system of $I(\Omega)$ given in the proof of Theorem~\ref{thm:kerA} is
 \begin{equation*}
   P:=\left\{p_{\ell}:\ell\in\{1,2\}^3\right\},\quad p_{\ell}(Z_1,Z_2):=\prod_{j=1}^3 (Z_{\ell_j}-z_{j,\ell_j}).
 \end{equation*}

 We start by illustrating the generic case that no two coordinates are equal,
 i.e.,~$z_{j,\ell}\ne z_{i,\ell}$ if $j\ne i$ and $\ell=1,2$.
 The zero locus of each individual polynomial $p_{\ell}$
 is illustrated in Figure~\ref{fig:3points},
 where each axis represents $\C$.
 The zero locus of each linear factor is a complex curve and illustrated by a single line.
 We note that the set $P$ is redundant,
 i.e.~the last three polynomials in the first row of Figure~\ref{fig:3points}
 are sufficient to recover the points uniquely as their common roots,
 % But in general it is not so obvious which polynomials can be omitted.
 but there is no obvious general rule which polynomials can be omitted.

 Four other point configurations are shown in Figure~\ref{fig:3points2}.
 In the first three configurations coordinates of different points agree,
 which allows to remove some polynomials from $P$.
 % \todo{As in the other case, right?~/U}
 In particular, the third point set which is collinear is generated
 already by $(Z_1-z_{1,1})(Z_1-z_{2,1})(Z_1-z_{3,1})$ and $(Z_2-z_{1,2})^3$.
 The fourth point set is generated either by the above set $P$ of polynomials
 or by $\prod_{j=1}^3(Z_1+Z_2-z_{j,1}-z_{j,2})$ and $Z_1-Z_2$,
 (which are not elements of $P$).
\end{example}

\begin{figure}[ht!]
 \begin{center}
\fbox{\begin{tikzpicture}[scale=0.85]
% dummy
\draw[-,very thick,draw=white] (-2,0)--(2,0);
% p_(1,1,1) = 0:
\draw[-,very thick,draw=blue] (-1,-2)--(-1,2);
\draw[-,very thick,draw=blue] (0,-2)--(0,2);
\draw[-,very thick,draw=blue] (1,-2)--(1,2);

% a_1
\node[circle,draw=black,fill=violet,minimum size=2mm,inner sep=0mm] at (-1,-1) {};
% a_2
\node[circle,draw=black,fill=violet,minimum size=2mm,inner sep=0mm] at (0,1) {};
% a_3
\node[circle,draw=black,fill=violet,minimum size=2mm,inner sep=0mm] at (1,0) {};
\end{tikzpicture}}
\fbox{\begin{tikzpicture}[scale=0.85]
% p_(1,1,2) = 0:
\draw[-,very thick,draw=blue] (-1,-2)--(-1,2);
\draw[-,very thick,draw=blue] (0,-2)--(0,2);
\draw[-,very thick,draw=blue] (-2,0)--(2,0);

% a_1
\node[circle,draw=black,fill=violet,minimum size=2mm,inner sep=0mm] at (-1,-1) {};
% a_2
\node[circle,draw=black,fill=violet,minimum size=2mm,inner sep=0mm] at (0,1) {};
% a_3
\node[circle,draw=black,fill=violet,minimum size=2mm,inner sep=0mm] at (1,0) {};
\end{tikzpicture}}
\fbox{\begin{tikzpicture}[scale=0.85]
% p_(1,2,1) = 0:
\draw[-,very thick,draw=blue] (-1,-2)--(-1,2);
\draw[-,very thick,draw=blue] (-2,1)--(2,1);
\draw[-,very thick,draw=blue] (1,-2)--(1,2);

% a_1
\node[circle,draw=black,fill=violet,minimum size=2mm,inner sep=0mm] at (-1,-1) {};
% a_2
\node[circle,draw=black,fill=violet,minimum size=2mm,inner sep=0mm] at (0,1) {};
% a_3
\node[circle,draw=black,fill=violet,minimum size=2mm,inner sep=0mm] at (1,0) {};
\end{tikzpicture}}
\fbox{\begin{tikzpicture}[scale=0.85]
% p_(2,1,1) = 0:
\draw[-,very thick,draw=blue] (-2,-1)--(2,-1);
\draw[-,very thick,draw=blue] (0,-2)--(0,2);
\draw[-,very thick,draw=blue] (1,-2)--(1,2);

% a_1
\node[circle,draw=black,fill=violet,minimum size=2mm,inner sep=0mm] at (-1,-1) {};
% a_2
\node[circle,draw=black,fill=violet,minimum size=2mm,inner sep=0mm] at (0,1) {};
% a_3
\node[circle,draw=black,fill=violet,minimum size=2mm,inner sep=0mm] at (1,0) {};
\end{tikzpicture}}
\fbox{\begin{tikzpicture}[scale=0.85]
% p_(1,2,2) = 0:
\draw[-,very thick,draw=blue] (-1,-2)--(-1,2);
\draw[-,very thick,draw=blue] (-2,1)--(2,1);
\draw[-,very thick,draw=blue] (-2,0)--(2,0);

% a_1
\node[circle,draw=black,fill=violet,minimum size=2mm,inner sep=0mm] at (-1,-1) {};
% a_2
\node[circle,draw=black,fill=violet,minimum size=2mm,inner sep=0mm] at (0,1) {};
% a_3
\node[circle,draw=black,fill=violet,minimum size=2mm,inner sep=0mm] at (1,0) {};
\end{tikzpicture}}
\fbox{\begin{tikzpicture}[scale=0.85]
% p_(2,1,2) = 0:
\draw[-,very thick,draw=blue] (-2,-1)--(2,-1);
\draw[-,very thick,draw=blue] (0,-2)--(0,2);
\draw[-,very thick,draw=blue] (-2,0)--(2,0);

% a_1
\node[circle,draw=black,fill=violet,minimum size=2mm,inner sep=0mm] at (-1,-1) {};
% a_2
\node[circle,draw=black,fill=violet,minimum size=2mm,inner sep=0mm] at (0,1) {};
% a_3
\node[circle,draw=black,fill=violet,minimum size=2mm,inner sep=0mm] at (1,0) {};
\end{tikzpicture}}
\fbox{\begin{tikzpicture}[scale=0.85]
% p_(2,2,1) = 0:
\draw[-,very thick,draw=blue] (-2,-1)--(2,-1);
\draw[-,very thick,draw=blue] (-2,1)--(2,1);
\draw[-,very thick,draw=blue] (1,-2)--(1,2);

% a_1
\node[circle,draw=black,fill=violet,minimum size=2mm,inner sep=0mm] at (-1,-1) {};
% a_2
\node[circle,draw=black,fill=violet,minimum size=2mm,inner sep=0mm] at (0,1) {};
% a_3
\node[circle,draw=black,fill=violet,minimum size=2mm,inner sep=0mm] at (1,0) {};
\end{tikzpicture}}
\fbox{\begin{tikzpicture}[scale=0.85]
% dummy
\draw[-,very thick,draw=white] (0,-2)--(0,2);
% p_(2,2,2) = 0:
\draw[-,very thick,draw=blue] (-2,-1)--(2,-1);
\draw[-,very thick,draw=blue] (-2,1)--(2,1);
\draw[-,very thick,draw=blue] (-2,0)--(2,0);

% a_1
\node[circle,draw=black,fill=violet,minimum size=2mm,inner sep=0mm] at (-1,-1) {};
% a_2
\node[circle,draw=black,fill=violet,minimum size=2mm,inner sep=0mm] at (0,1) {};
% a_3
\node[circle,draw=black,fill=violet,minimum size=2mm,inner sep=0mm] at (1,0) {};
\end{tikzpicture}}
\end{center}

  \caption{Zero sets of the polynomials in $P$, $d=2$, $M=3$, and for the case that no two coordinates are equal.}
  \label{fig:3points}
 \end{figure}
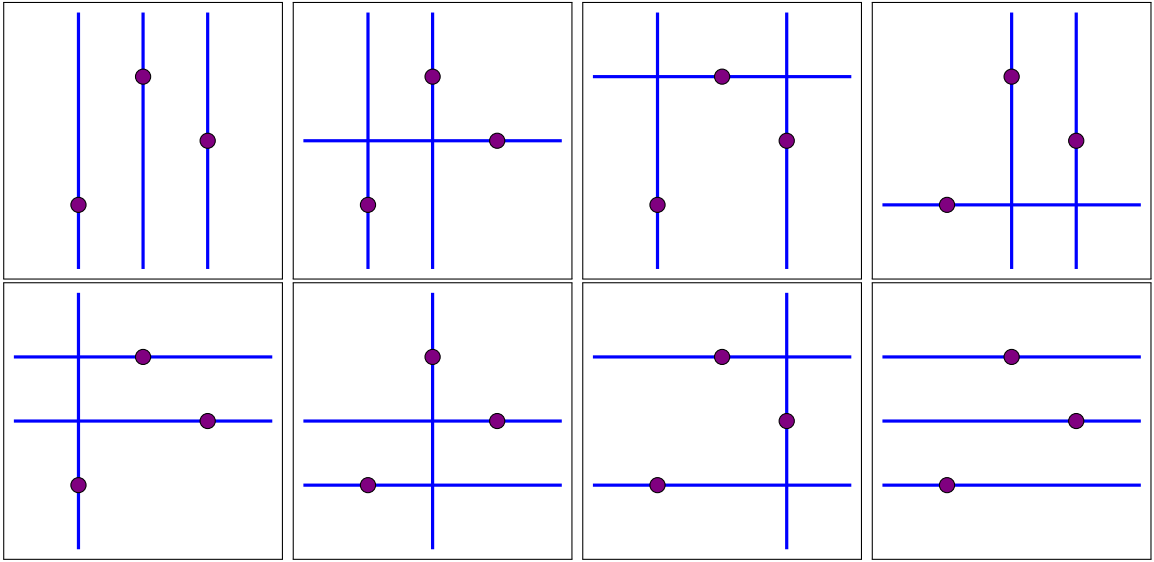

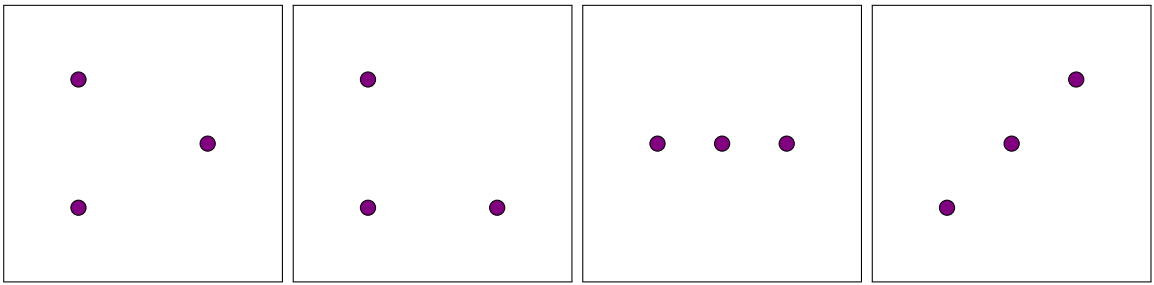
\begin{figure}[ht!]
 \begin{center}
\fbox{\begin{tikzpicture}[scale=0.85]
% dummy
\draw[-,very thick,draw=white] (-2,-2)--(2,2);
% a_1
\node[circle,draw=black,fill=violet,minimum size=2mm,inner sep=0mm] at (-1,-1) {};
% a_2
\node[circle,draw=black,fill=violet,minimum size=2mm,inner sep=0mm] at (-1,1) {};
% a_3
\node[circle,draw=black,fill=violet,minimum size=2mm,inner sep=0mm] at (1,0) {};
\end{tikzpicture}}
\fbox{\begin{tikzpicture}[scale=0.85]
% dummy
\draw[-,very thick,draw=white] (-2,-2)--(2,2);
% a_1
\node[circle,draw=black,fill=violet,minimum size=2mm,inner sep=0mm] at (-1,-1) {};
% a_2
\node[circle,draw=black,fill=violet,minimum size=2mm,inner sep=0mm] at (-1,1) {};
% a_3
\node[circle,draw=black,fill=violet,minimum size=2mm,inner sep=0mm] at (1,-1) {};
\end{tikzpicture}}
\fbox{\begin{tikzpicture}[scale=0.85]
% dummy
\draw[-,very thick,draw=white] (-2,-2)--(2,2);
% a_1
\node[circle,draw=black,fill=violet,minimum size=2mm,inner sep=0mm] at (-1,0) {};
% a_2
\node[circle,draw=black,fill=violet,minimum size=2mm,inner sep=0mm] at (0,0) {};
% a_3
\node[circle,draw=black,fill=violet,minimum size=2mm,inner sep=0mm] at (1,0) {};
\end{tikzpicture}}
\fbox{\begin{tikzpicture}[scale=0.85]
% dummy
\draw[-,very thick,draw=white] (-2,-2)--(2,2);
% a_1
\node[circle,draw=black,fill=violet,minimum size=2mm,inner sep=0mm] at (-1,-1) {};
% a_2
\node[circle,draw=black,fill=violet,minimum size=2mm,inner sep=0mm] at (0,0) {};
% a_3
\node[circle,draw=black,fill=violet,minimum size=2mm,inner sep=0mm] at (1,1) {};
\end{tikzpicture}}
\end{center}
  \caption{Point sets $\Omega\subset\C^2$, $d=2$, $M=3$.}
  \label{fig:3points2}
 \end{figure}

 \begin{remark}\label{rem:nGeM}
  Concerning the ``natural'' generator
  $P=\left\{\prod_{j=1}^M(Z_{\ell_j}-z_{j,\ell_j}):\ell_j\in\{1,\dots,d\}\right\}$
  used in the proof above,
  we note that although the ideals
  $\langle P\rangle=\langle\ker\calA_M\rangle$ coincide,
  the subvector~space inclusion
  \begin{equation*}
   \sspan P\subset\ker\calA_M
  \end{equation*}
  is strict in general
  as can be seen for $d=2$, $M=2$,
  $z_1=(0,0)$, $z_2=(1,1)$
  and the polynomial $Z_1-Z_2\in\ker\calA_2$.
  Moreover,
  we have the cardinality $\lvert P\rvert=d^M$,
  at least for different coordinates $z_{j,\ell}$,
  and thus $\lvert P\rvert\gg M^d-M=\dim\ker\calA_M$,
  i.e.,~the generator $P$ contains many linear dependencies
  and is highly redundant for large $M$.

  Finally, we would like to comment on the degree $n$
  and the total number of samples $(2n+1)^d$
  with respect to the number of parameters $M$:
  \begin{enumerate}
   \item
    A small degree $n\in\N$, $M<N<M+d$,
    and surjective $\calA_n$ results in an uncountably infinite zero locus
    $V(\ker{\calA}_n)$, since $\dim(\ker{\calA}_n)\le N-M<d$
    and thus $I(\Omega)$ is generated by less than $d$ polynomials.
    % \todo{References.~/U}
   \item
    Increasing the degree results ``generically'' in a finite zero locus, cf.~\cite{AnCaHo10},
    but ``generically'' identifies spurious parameters since e.g.~for
    $d=2$ B\'ezout's theorem yields $\lvert V(p,q)\rvert\le\deg(p)\deg(q)$
    with equality in the projective setting
    (counting the roots with multiplicity),
    for coprime polynomials $p,q\in\ker{\calA}_n$.
  \end{enumerate}
\end{remark}
\begin{remark}
We discuss a slight modification of our approach.
Instead of $I_n=\{0,\dots,n\}^d=\{k\in\NZ^d:\norm{k}_\infty\le n\}$
we take $J_n:=\{k\in\NZ^d:\norm{k}_1\le n\}$
as index set
and consider the matrix
\begin{equation*}
H_n(f):=(f(k+\ell))_{k,\ell\in J_n}\in\C^{\binom{n+d}{d}\times\binom{n+d}{d}}
\end{equation*}
instead of $T_n(f)=(f(k-\ell))_{k,\ell\in I_n}$.
Theorem~\ref{thm:kerA} also holds with $T_n$ replaced by $H_n$
with almost no change to the proof.
In this way we need only $\binom{2n+d}{d}$ rather than $(2n+1)^d$ samples of $f$
and also allow for arbitrary parameters $z_j\in\C^d$ instead of $z_j\in\C_\ast^d$.
While $T_n$ is a multilevel Toeplitz matrix, $H_n$ is a submatrix of a multilevel Hankel matrix, and
for the trigonometric setting discussed in the following subsection,
it is more natural to consider the moments $f(k)$, $k\in\Z^d$, $\norm{k}_\infty\le n$,
than $f(k)$, $k\in\NZ^d$, $\norm{k}_1\le2n$.
\end{remark}

\subsection{Parameters on the torus, trigonometric polynomials, and stable solution}\label{subsect:trigonometric}

We now restrict our attention to parameters $z_j\in\T^d$,
hence $z_j=\eip{t_j}$ for a unique $t_j\in[0,1)^d$.
In this case, $V(\ker\calA_n)$ fulfills a $2^d$-fold symmetry in the following sense.
  Let $p(z)=\sum_{k=0}^n \hat p_k z^k\in\ker \calA_n$ and $z=(z_1,\hdots,z_d)^{\top} \in \C^d$ with $p(z)=0$, then
  $z'=(\overline{z_1}^{-1},z_2,\hdots,z_d)^{\top}$ is a root of the 1st-coordinate conjugate reciprocal polynomial
  \begin{equation*}
   q(z):=\overline{\overline{z_1}^n p(\overline{z_1}^{-1},z_2,\hdots,z_d)}=\sum_{k=0}^n \overline{\hat p_{n-k_1,k_2,\hdots,k_d}} z^k.
  \end{equation*}
 Since the roots $z\in\Omega\subset\T^d$ are self reciprocal $z'=z$, we have $q\in\ker\calA_n$ and thus
 $z\in V(\ker\calA_n)$ implies $z'\in V(\ker\calA_n)$ for all choices of a conjugated reciprocal coordinate.

  Moreover, we have the following construction of a so-called dual certificate~\cite{CaFe14,CaFe13,BeDeFe15a,BeDeFe15b}.
\begin{thm}\label{thm:DualCertificate}
 Let $d,n,M\in\N$, $n\ge M$, $t_j\in [0,1)^d$, $j=1,\hdots,M$, $z_j:=\eip{t_j}$, and $\Omega:=\{z_j:j=1,\hdots,M\}$ be given.
 Moreover, let $\hat p_{\ell}\in\C^N$, $\ell=1,\hdots,N$, be an orthonormal basis with $\hat p_{\ell}\in\ker(T_n)^{\bot}$, $\ell=1,\hdots,M$,
 and $p_{\ell}:\C_\ast^d\rightarrow\C$, $p_{\ell}(z)=\sum_{k\in I_n} \hat p_{\ell,k} z^k$, then $p:[0,1)^d\rightarrow\C$,
 \begin{equation}
	 \label{SummeVonQuadraten}
  p(t)=\frac{1}{N} \sum_{\ell=1}^M |p_{\ell}(\eip{t})|^2,
 \end{equation}
 is a trigonometric polynomial of degree $n$ and fulfills $0\le p(t) \le 1$ for all $t\in [0,1)^d$ and $p(t)=1$ if and only if $t=t_j$ for some $j=1,\hdots,M$.
\end{thm}
\begin{proof}
 First note that every orthonormal basis $\hat p_{\ell}\in\C^N$, $\ell=1,\hdots,N$, leads to
 \begin{equation*}
  \sum_{\ell=1}^N |p_{\ell}(z)|^2
  %&= \sum_{\ell=1}^N \sum_{r,s=1}^N \overline{\hat p_{\ell,r} z^r} \hat p_{\ell,s} z^s\\
  = \sum_{r,s=1}^N \overline{z^r} z^s \sum_{\ell=1}^N  \overline{\hat p_{\ell,r}} \hat p_{\ell,s}\\
  = \sum_{r=1}^N |z^r|^2=N
 \end{equation*}
 for $z\in\T^d$. Moreover, $\bar z=z^{-1}$ on $\T^d$ yields that $p$ is indeed a trigonometric polynomial.
 Finally, Theorem~\ref{thm:kerA} assures $\sum_{\ell=M+1}^N |p_{\ell}(z)|^2=0$ if and only if $z\in\Omega$.
\end{proof}
We proceed with an estimate on the condition number of the preconditioned matrix $T=T_n$.
\begin{definition}
Let $M\in\N$ and $\Omega=\{\eip{t_j}:t_j\in [0,1)^d,\;j=1,\hdots,M\}$, then
\begin{equation*}
\sep(\Omega):=\min_{r\in\Z^d,\;j\ne\ell}\| t_j-t_{\ell}+r\|_{\infty}
\end{equation*}
is the \emph{separation distance} of $\Omega$.
For $q>0$, we say that $\Omega$ is \emph{$q$-separated} if $\sep(\Omega)>q$.
\end{definition}

\begin{thm}\label{thm:condT}
 Let $d,n,M\in\N$, $t_j\in [0,1)^d$, $j=1,\hdots,M$, $z_j:=\eip{t_j}$,
 $q>0$,
 and
 $\Omega:=\{z_j:j=1,\hdots,M\}$ be $q$-separated. Moreover, let $\hat f_j>0$, then $n\ge2d q^{-1}$ implies the condition number estimate
 \begin{equation*}
  \cond_2 W T W \le \frac{(nq)^{d+1}+(2d)^{d+1}}{(nq)^{d+1}-(2d)^{d+1}} \cdot \frac{\max_j \hat f_j}{\min_j \hat f_j}\text{,}
 \end{equation*}
 where the diagonal preconditioner $W=\diag w$, $w_k>0$, $k\in I_n$, is well chosen.
 In particular, $\lim_{n\rightarrow\infty} \cond_2 W T W=\max_j \hat f_j/\min_j \hat f_j$.
\end{thm}
\begin{proof}
 First note, that the matrix $T$ is hermitian positive semidefinite and define the condition number by $\cond_2 T:=\|T\|_2\|T^{\dagger}\|_2$, where
 $T^{\dagger}$ denotes the Moore-Penrose pseudoinverse.
 Let $D=\diag d$, $d_j=\hat f_j^{1/2}$, $j=1,\hdots,M$, and $K=A W^2 A^*\in\C^{M\times M}$, then we have
 \begin{equation*}
  \cond_2 W T W=\cond_2 W A^* D^2 A W =\cond_2 D A W^2 A^* D
  =\frac{\max \hat f_j}{\min \hat f_j}\cdot \cond_2 K
 \end{equation*}
 and Corollary~4.7 in~\cite{KuPo07} yields the condition number estimate. The second claim follows since
 $\lim_{n\rightarrow\infty} \cond_2 K=1$.
\end{proof}

In summary, the condition
\begin{equation}\label{eq:n}
 n\ge \max\{2d q^{-1},M\}
\end{equation}
allows for unique reconstruction of the parameters $\Omega$ and stability is guaranteed when computing the kernel polynomials from the given moments.

\begin{remark}\label{rem:nGeQ}
 Up to the constant $2d$, the condition $n>2d/q$ in the assumption of Theorem~\ref{thm:condT} is optimal in the sense that
 equidistant nodes $t_j=j/m$, $j\in I_m$, $n<q^{-1}=m$, imply $A\in\C^{m^d\times n^d}$ and $\rank A=n^d<m^d=M$.
 We expect that the constant $2d$ can be improved and indeed, a discrete variant of Ingham's inequality~\cite{KoLo04},~\cite[Lemma~2.1]{PoTa132}
 replaces $2d$ by $C\sqrt{d}$ but gives no explicit estimate on the condition number.

 Moreover, Theorem~\ref{thm:kerA} asserts that the condition on the degree $n$ with respect to the number ob parameters $M$ is close to optimal
 in the specific setting.
 We briefly comment on the following typical scenarios for the point set $\Omega$ and the relation~\eqref{eq:n}:
 \begin{enumerate}
  \item quasi-uniform parameters $t_j\in[0,1)^d$ might be defined via $\sep(\Omega) \approx C_d M^{-1/d}$, i.e., $\max\{2d q^{-1},M\}=M$,
  \item equidistant and co-linear parameters, e.g. $t_j=M^{-1}(j,\hdots,j)^{\top}$, imply $\sep(\Omega) \approx C M^{-1}$, i.e., both terms are of similar size,
  \item and finally parameters $t_j\in[0,1)^d$ chosen at random from the uniform distribution, imply $\mathbb{E}\sep(\Omega)=C'_d M^{-2}$,
  see e.g.~\cite{ReScTh15}, and thus $\max\{2d q^{-1},M\}=C''_d M^2$.
 \end{enumerate}
\end{remark}

Dropping the condition $n>M$ in~\eqref{eq:n} and restricting to the torus, we still get the following result on how much the
roots of the polynomials in the kernel of $T$ can deviate from the original set $\Omega$.

\begin{thm}\label{thm:stabVkerT}
 Let $d,n,M\in\N$, $t_j\in [0,1)^d$, $j=1,\hdots,M$, $z_j:=\eip{t_j}$,
 $q>0$, and $\Omega:=\{z_j:j=1,\hdots,M\}$ be $q$-separated, then $n\ge2d q^{-1}$ implies
 \begin{equation*}
  \Omega  \subset  V(\ker T)\cap\T^d \subset \{z\eip{t}:z\in\Omega, \;\|t\|_{\infty}<2d/n\}.
 \end{equation*}
\end{thm}
\begin{proof}
 We prove the assertion by contradiction. Let $y\in V(\ker T)\cap\T^d$ be $2d/n$-separated from the point set $\Omega\subset\T^d$ and
 let $\hat p_{\ell}\in\C^{N}$, $\ell=1,\hdots,N-M$, constitute a basis of $\ker T$.
 By definition, we have $p_{\ell}(y)=0$ and thus the augmented Fourier Matrix
 \begin{equation*}
  A_y:=\pmat{A\\ e_y},\quad e_y:=(\eip{ky})_{k\in I_n}^{\top},
 \end{equation*}
 fulfills $A_y \hat p_{\ell}=0$, i.e., $\dim\ker A_y\ge N-M$. On the other hand, Corollary~4.7 in~\cite{KuPo07} implies $\rank A_y=M+1$
 and thus the contradiction $N=\dim \ker A_y + \rank A_y \ge N+1$.
\end{proof}

\subsection{Prototypical algorithm}\label{subsect:algorithm}

Let $f$ be an $M$-sparse $d$-variate exponential sum
with pairwise distinct parameters $z_j\in\C_\ast^d$
and $n\ge M$ be an upper bound.
Theorem~\ref{thm:kerA} justifies the following
prototypical formulation of the multivariate Prony method.
\begin{algorithm}[H]
  \begin{tabular}{ll}
  Input: & $d,n\in\N$,\\
         & $f(k)$, $k\in\{-n,\hdots,n\}^d$
  \end{tabular}
  \begin{algorithmic}
    \itemsep=1.1ex
    \STATE Set up $T_n=\left(f(k-\ell)\right)_{k,\ell\in I_n}\in\C^{N\times N}$
    \STATE Compute $\ker T_n$
    \STATE Compute $V(\ker T_n)$
  \end{algorithmic}

  Output: $V(\ker T_n)=\{z_1,\dots,z_M\}$\rule[-1.5ex]{0ex}{3ex}
  \caption{Multivariate Prony method.}
  \label{alg:prony2d}
\end{algorithm}
The third step,
i.e.,~the computation of the zero locus $V(\ker T_n)$,
is beyond the scope of this paper
and several methods can be found elsewhere,
see e.g.~\cite{Bertini,MoSa00,SoBaLa14,Stetter_Numerical-Polynomial-Algebra}.
We further note that the number $(2n+1)^d$ of used samples scales as $\OO{M^d}$
and that standard algorithms for computing the kernel of the matrix $T_n$ have cubic complexity.

\section{Other approaches}
\label{sect:ErstesKapitel}

There are many variants of the one dimensional moment problem from Section~\ref{sect:Einleitung}, originating from such diverse fields as for example signal processing, electro engineering, and quantum chemistry, with as widespread applications as spectroscopy, radar imaging, or super-resolved optical microscopy, see e.g.~the survey paper~\cite{PlTa14}.
Variants of Prony's method with an increased stability or a direct computation of the parameters without the detour via polynomial coefficients include
for example MUSIC~\cite{Sc86}, ESPRIT~\cite{RoKa90}, the Matrix-Pencil method~\cite{HuSa90}, the Approximate Prony method~\cite{PoTa10}, the Annihilating Filter method~\cite{VeMaBl02},
and methods relying on orthogonal polynomials~\cite{FiMhPr12}.

Multivariate generalizations of these methods have been considered in~\cite{JiSiBe01,AnCaHo10} by realizing the parameters as common roots of multivariate polynomials.
In contrast to our approach, both of these papers have an emphasis on the generic situation where e.g.~the zero locus of two bivariate polynomials is finite.
In this case, the total number of used moments for reconstruction might indeed scale as the number of parameters but no guarantee is given for a specific instance of the moment problem.
A second line of multivariate generalizations~\cite{PoTa132,PlWi13} decomposes the multivariate moment problem into a series of univariate moment problems via projections of the measure.
While again this approach typically works well, the necessary number of a-priori chosen projections for a signed measure scales as the number of parameters in the bivariate case~\cite{DiIs15}.
We note that the subset
\begin{equation*}
  P_0:=\{\prod_{j=1}^M(Z_{\ell}-z_{j,\ell}):\ell=1,\hdots,d\}\subset I(\Omega),
\end{equation*}
of the set of generators in the proof of Theorem~\ref{thm:kerA} are exactly the univariate polynomials when projecting onto the $d$ coordinate axes, see also the
first and last zero locus in Figure~\ref{fig:3points}.

A different approach to the moment problem from Section~\ref{sect:Einleitung} has been considered in~\cite{CaFe14,CaFe13,BeDeFe15a,BeDeFe15b} and termed `super-resolution'.
From a signal processing perspective, knowing the first moments is equivalent to sampling a low-pass version of the measure and restoring the high frequency information from these samples.
With the notation of Remark~\ref{rem:trigMoments} the measure $\tau$ with parameters $t_j\in[0,1)^d$ is the unique minimizer of
\begin{equation*}
 \min \|\nu\|_{\mathrm{TV}} \quad \text{s.t.} \quad \int_{[0,1)^d}\eip{kt}\dd\nu(t)=f(k), \; k\in I_n,
\end{equation*}
provided the parameters fulfill a separation condition as in Section~\ref{subsect:trigonometric}.
This is proven via the existence of a so-called dual certificate~\cite[Appendix~A]{CaFe14} and becomes computationally attractive by recasting this dual problem as a semidefinite program.
The program has just $(n+1)^2/2$ variables in the univariate case~\cite[Corollary~4.1]{CaFe14}, but at least we do not know an explicit bound on the number of variables in the multivariate case, see~\cite[Remark~4.17, Theorem~4.24, and Remark~4.26]{Du07}.

Finally note that there is a large body of literature on the related topic of reconstructing a multivariate sparse trigonometric polynomials from samples, see e.g.~\cite{BeTi88,Ma95,GiLaLe02,CaRoTa06,Ra06,RuVe08,GiInIwSchm14}.
Translated to the situation at hand, all these methods heavily rely on the fact that the parameters $t_j\in[0,1)^d$ are located on a Cartesian grid with
mesh sizes $1/m_1,\hdots,1/m_d$ for some $m_1,\hdots,m_d\in\N$ and deteriorate if this condition fails~\cite{ChScPeCa11}.
Hence, these methods lack one major advantage of Prony's method, namely that the parameters $t_j\in[0,1)^d$ can, in principle, be reconstructed with infinite precision.

\goodbreak
%%%%%%%%%%%%%%%%%%%%%%%%%%%%%%%%%%%%%%%%%%%%%%%%%%%%%%%%%%%%%%%%%%%%%%%%%%%%%%
\section{Numerical results}\label{sect:num}
 All numerical experiments are realized in MATLAB~2014a on an Intel~i7, 12GByte, 2.1GHz, Ubuntu 14.04.

\begin{example}[$d=1$]
 \begin{figure}[htbp]
\centering
  \subfigure[Sum of squared absolute values of kernel polynomials on $\T$ (identified with $[0,1)$).]
  {\includegraphics[width=0.3\textwidth]{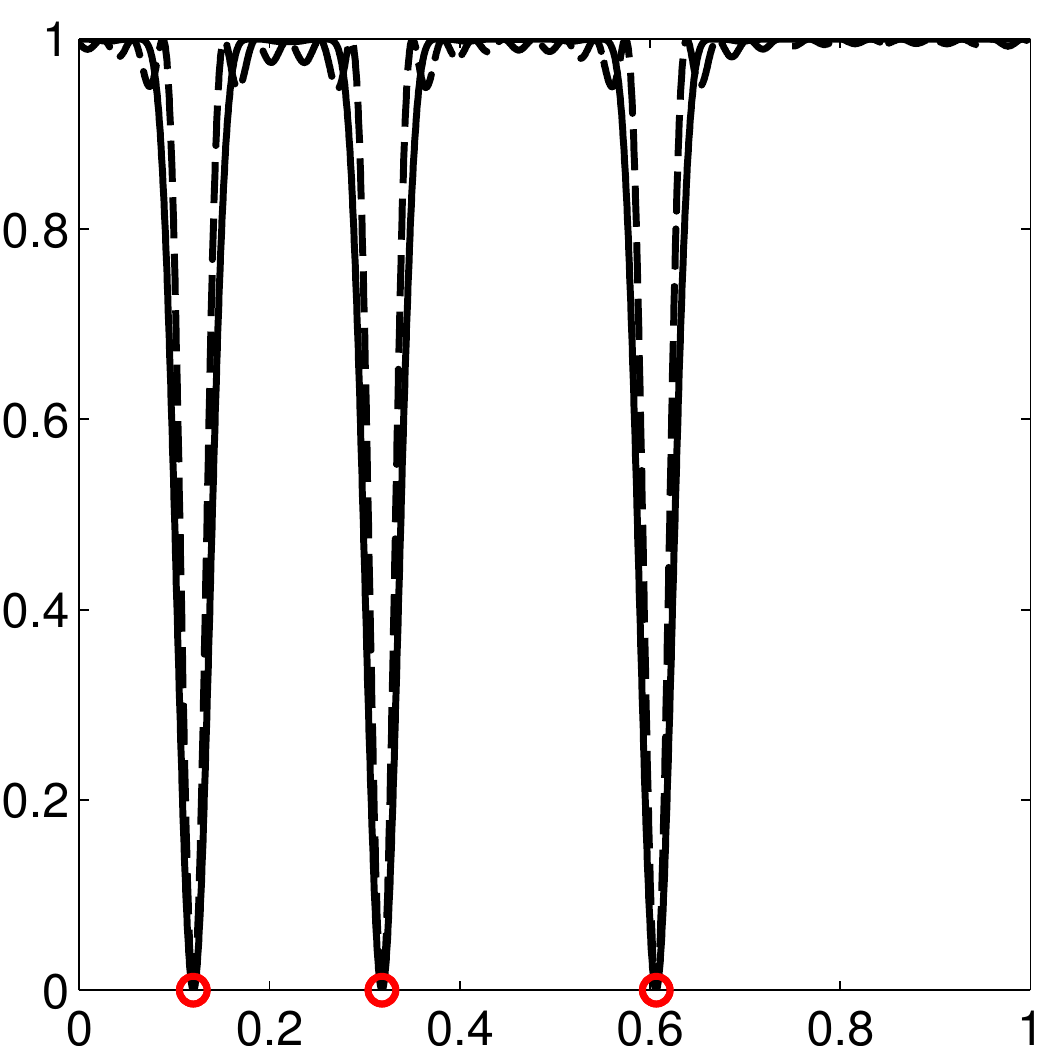}}
  \subfigure[Zero set of sum of squared absolute values of kernel polynomials on $\C$.]
  {\includegraphics[width=0.3\textwidth]{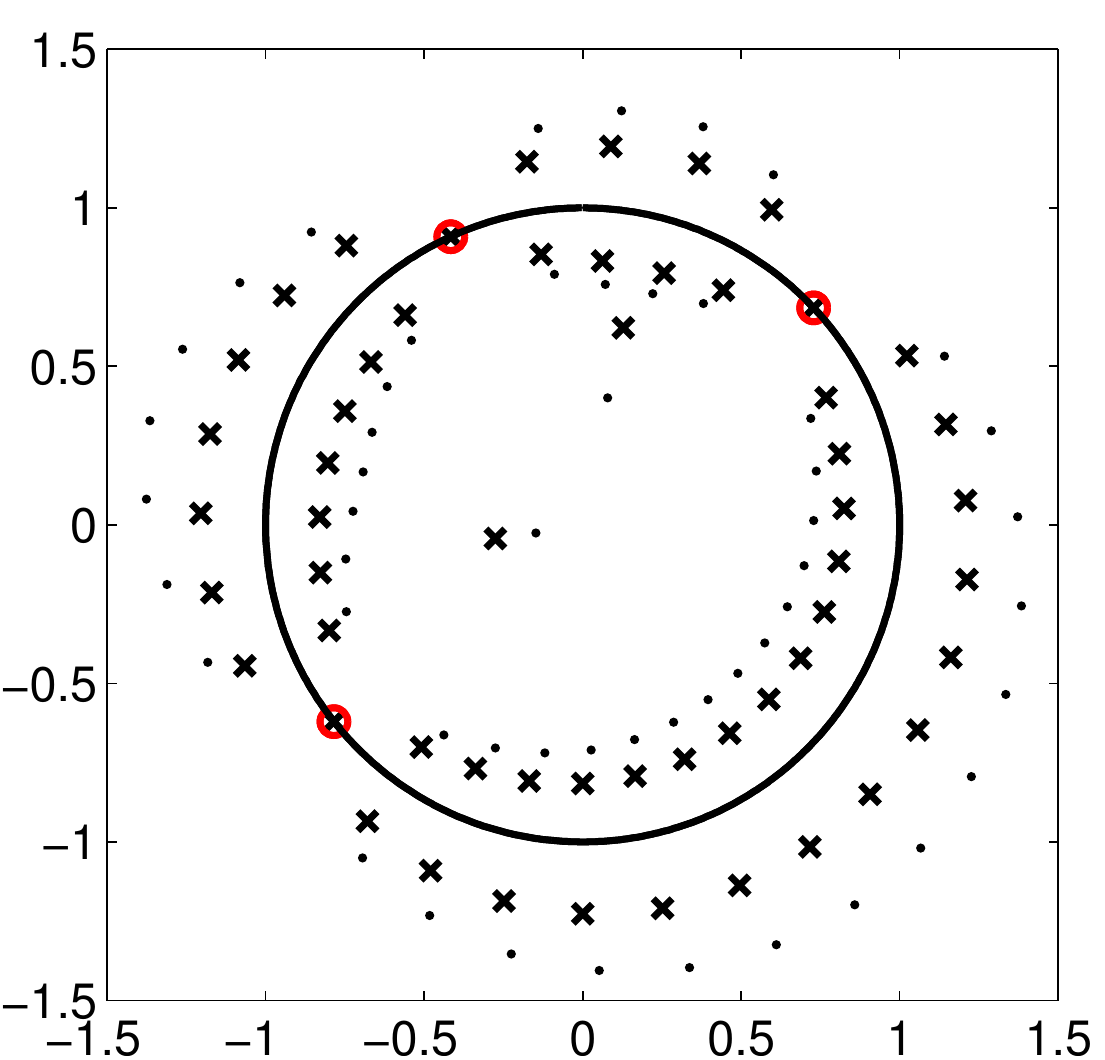}}
  \subfigure[Squared absolute values of the first polynomial orthogonal to the kernel.]
  {\includegraphics[width=0.3\textwidth]{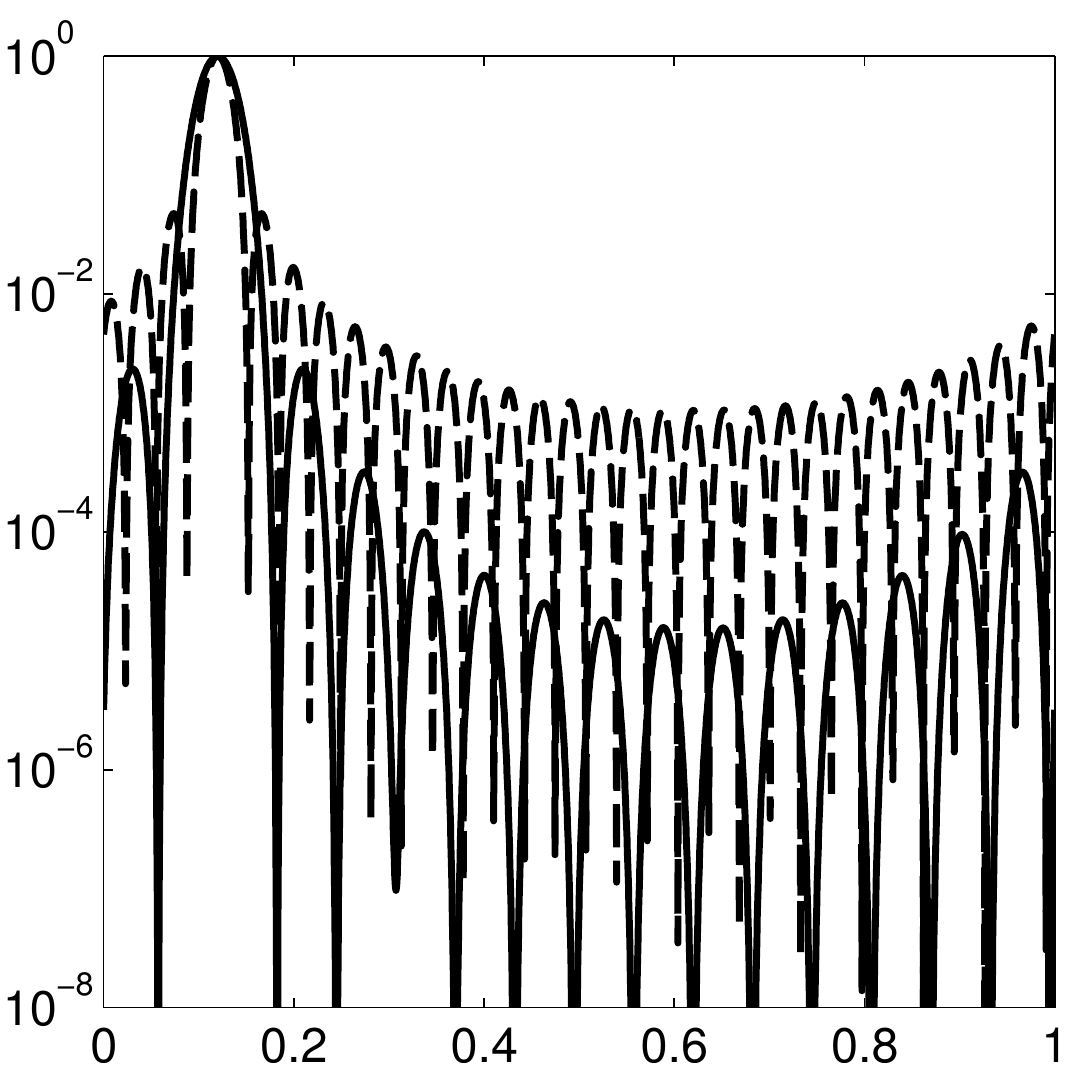}}
\caption{Parameters $d=1$, $M=3$, $n=30$,
$t_1=0.12$,
$t_2=1/\pi$,
and
$t_3=\e^{-1/2}$.
Dashed lines and $\times$ indicate no weighting,
solid lines and $\cdot$ indicate triangular weights $w_k=\min\{k+1,n-k\}$, $k=0,\hdots,n-1$.}
\label{fig:d=1}
\end{figure}
We consider the case $d=1$
with parameters on the $1$-torus $\T$ that we identify with the interval $[0,1)$.
For a $3$-sparse exponential sum
some of the associated (trigonometric) polynomials
are visualized in Figure~\ref{fig:d=1},
where we start with the upper bound $n=30\ge3$
and also indicate the effects of a preconditioner $W$ according to Theorem~\ref{thm:condT}
on the roots of the polynomials.

The method introduced in~\cite{CaFe14} finds a polynomial of the form~\eqref{SummeVonQuadraten}
as a solution to a convex optimization problem,
whereas we find
% it % Not the same polynomial, is it? /U
such a polynomial
with Prony's method.
For this comparison we used the MATLAB code provided in~\cite{CaFe14} and modified it
so that it runs for different problem sizes depending on the sparsity $M = 1,\dots,100$.
This means that we used roughly $5M$ samples and random parameters $t_j\in[0,1)$, $j = 1,\dots,M$,
satisfying the separation condition in~\cite{CaFe14}.
We only measured the time for finding a polynomial of the form~\eqref{SummeVonQuadraten},
since the calculation of the roots is basically the same in both algorithms.
In Figure~\ref{CandesVsProny}~(a),
where the times needed with~\textnormal{\texttt{cvx}} are depicted as circles
and the times needed by Prony's method are depicted as crosses,
we see that the solution via convex optimization takes considerably more time.
Note that the end criterion of the convex optimization program is set to roughly $10^{-6}$,
therefore the solution accuracy does not increase beyond this point,
whereas for Prony's method the solutions in this test are all in the order of machine accuracy,
$10^{-15}$.
 \begin{figure}[htbp]
\centering
\subfigure[Running time for Prony method~($+$)
% resp.~\textnormal{\texttt{cvx}}
resp.~super-resolution~($\circ$)
for varying number of parameters in the case $d=1$.]
  {\includegraphics[width=0.45\textwidth]{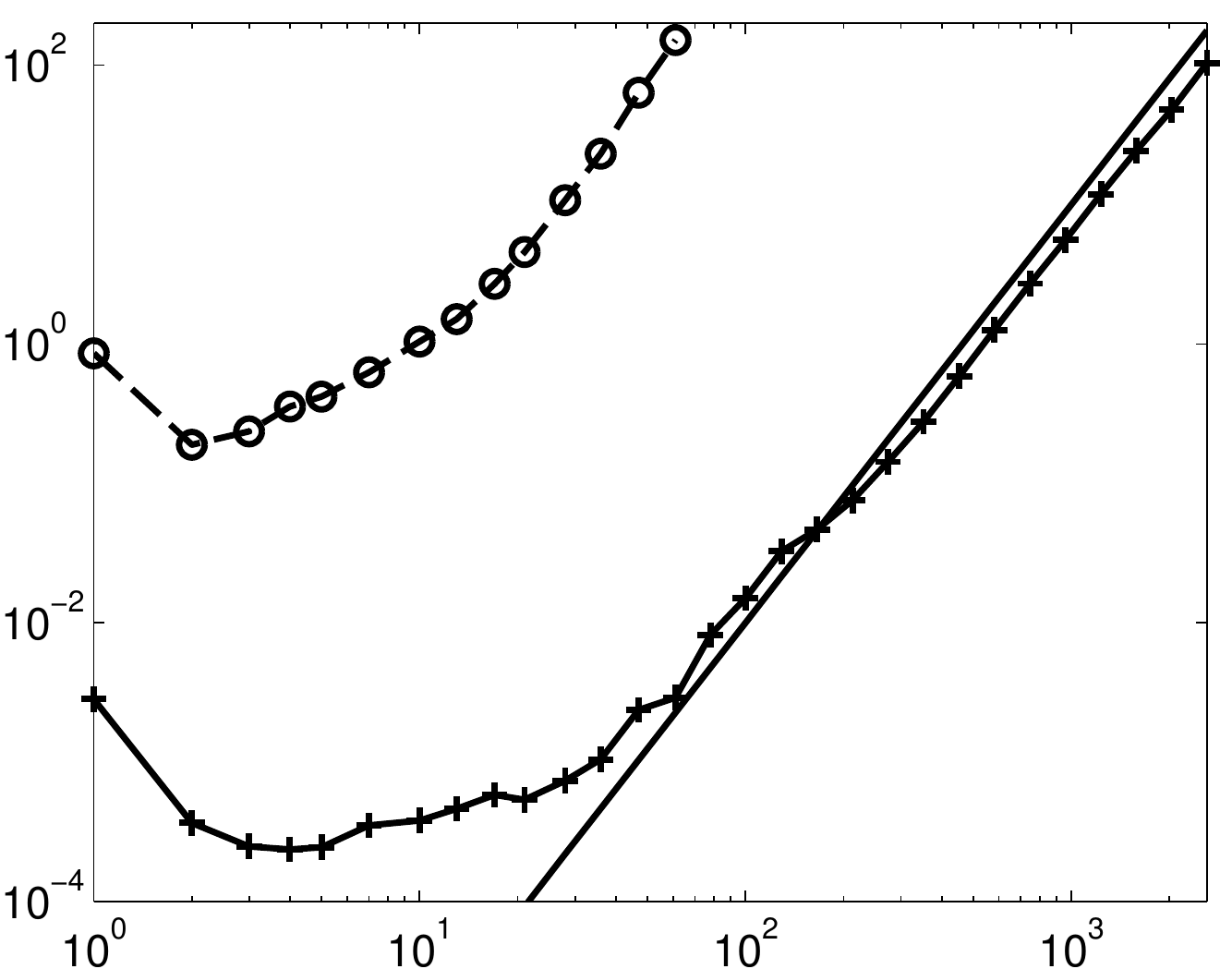}}
  \subfigure[Running time for Prony method~($+$)
% resp.~\textnormal{\texttt{cvx}}
resp.~super-resolution~($\circ$)
for varying number of parameters in the case $d=2$.]
  {\includegraphics[width=0.45\textwidth]{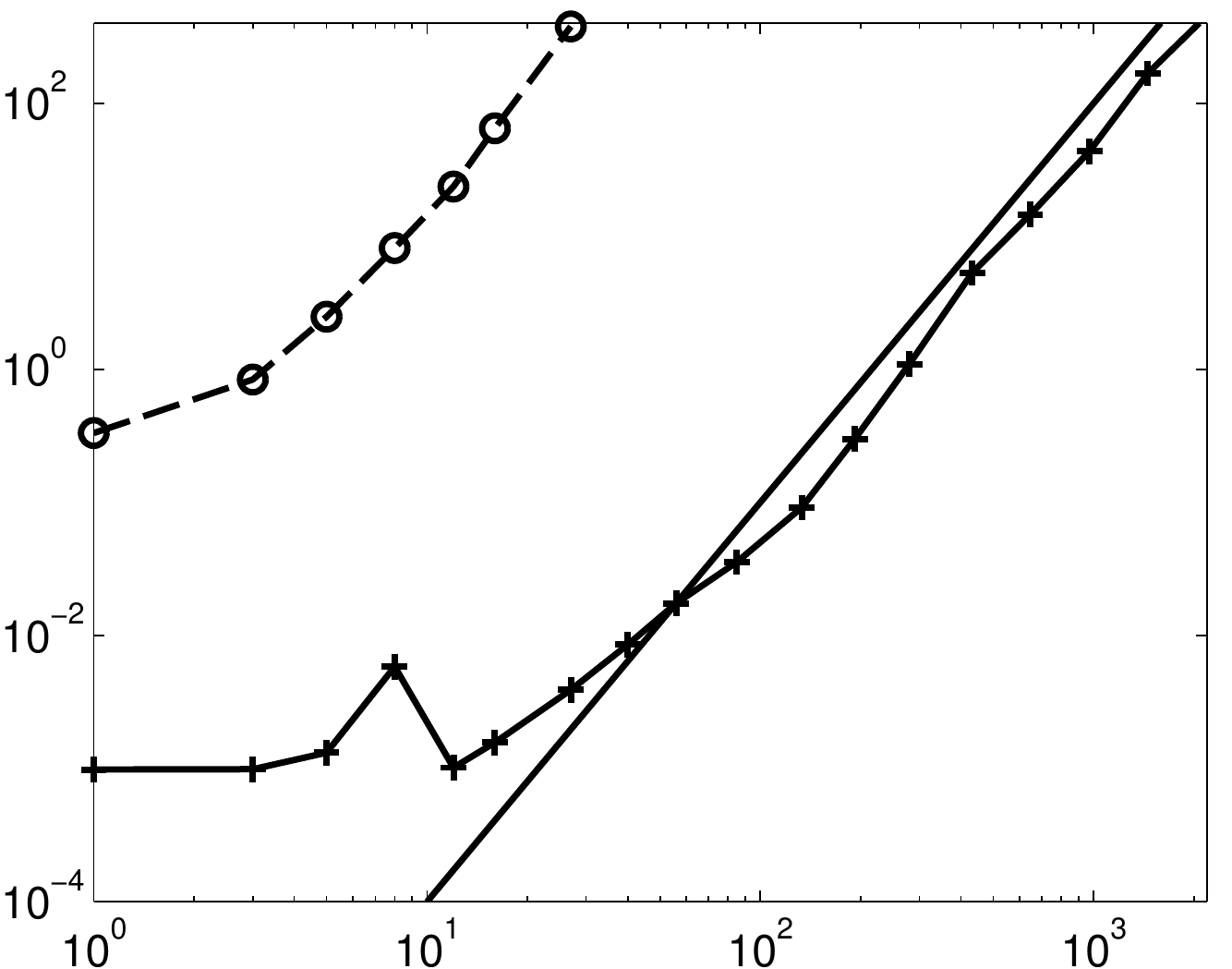}}
\caption{Time comparison for extracting a polynomial of form~\eqref{SummeVonQuadraten},
once with~\textnormal{\texttt{cvx}} depicted as circles and via Prony's method depicted as crosses.}
%Time comparison for extracting a polynomial of form~\eqref{SummeVonQuadraten} out of roughly $5M$ sampling points for different problem sizes depending on $M = 1,\dots,100$. The times needed to find such a polynomial as a solution of a convex optimization problem via \texttt{cvx} are depicted as dots, whereas the times spent for computing such a polynomial via Prony's method are depicted as crosses.}
\label{CandesVsProny}\end{figure}
\end{example}

\begin{example}[$d=2$]\label{ex:toy-example-2d}
We demonstrate our method to reconstruct the parameters from the moments
$f\colon\Z^2\to\C$,
$k\mapsto(1,1)^k+(-1,-1)^k$.
For moments of order $\lvert k\rvert\le n=2$ and the associated space of polynomials $\Pi_2$
with reverse lexicographical order on the terms,
we get the $9\times9$ block Toeplitz matrix $T=T_2$
with the Toeplitz blocks $T^\prime,T^{\prime\prime}$
as follows:
\begin{equation*}
 T=(f(k-\ell))_{k,\ell\in I_2}
  =\begin{pmatrix}
  T^\prime&T^{\prime\prime}&T^\prime\\
  T^{\prime\prime}&T^\prime&T^{\prime\prime}\\
  T^\prime&T^{\prime\prime}&T^\prime
\end{pmatrix},
\quad
T^\prime=\begin{pmatrix}
2&0&2\\
0&2&0\\
2&0&2
\end{pmatrix},
\quad
T^{\prime\prime}=\begin{pmatrix}
0&2&0\\
2&0&2\\
0&2&0
\end{pmatrix}.
\end{equation*}
A vector space basis of $\ker T$ is given by the polynomials
\begin{align*}
p_1&=-1+Z_1^2\text{, }\quad
&p_2&=-Z_1+Z_2\text{, }\quad
&p_3&=-1+Z_1Z_2\text{, }\\
p_4&=-Z_1+Z_1^2Z_2\text{, }\quad
&p_5&=-1+Z_2^2\text{, }\quad
&p_6&=-Z_1+Z_1Z_2^2\text{, }\\
p_7&=-1+Z_1^2Z_2^2\text{.}
\end{align*}
Since
$p_3=p_1+Z_1 p_2$,
$p_4=Z_1 p_3$,
$p_5=Z_2 p_2+p_3$,
$p_6=Z_1 p_5$,
and
$p_7=(1+Z_1Z_2)p_3$,
we have $\langle\ker T\rangle=\langle p_1,p_2\rangle$ and
hence $V(\ker T)=V(p_1,p_2)=\set{(1,1),(-1,-1)}$.
The zero loci of $p_1,p_2$ are depicted in
Figure~\ref{fig:toy-example-2d}~(a)
(in the style of Figure~\ref{fig:3points})
resp.~(b),
where the torus $\T^2$ is identified with $[0,1)^2$.
Note that we would typically expect the intersection
of the zero locus of each polynomial with the torus to be finite,
which is the case neither for~$p_1$ nor~$p_2$.
In Figure~\ref{fig:toy-example-2d}~(c)
the sum of the squared absolute values of an orthonormal basis
of $\ker T$ is drawn.
\begin{figure}[htbp]
\centering
\hfill
\subfigure[][See also Fig.~\ref{fig:3points}, zero loci $V(p_1),V(p_2)\subset\C^2$.]{
%  \protect\label{subfig:toy-example-2d-pictorial}
  %\fbox{\begin{tikzpicture}[scale=\textwidth/(4*5cm)]
  \fbox{\begin{tikzpicture}[scale=0.85]
% dummy
\draw[-,very thick,draw=white] (-2,-2)--(2,2);

% Should there be a "coordinate system" in the picture?
% Without one the difference between "pictorial" and "torus" is unclear.
% \draw[->,very thick] (-2.5,0)--(2.5,0);
% \draw[->,very thick] (0,-2)--(0,2);

% p_1 = 0:
\draw[-,very thick,draw=blue] (1,-2)--(1,2);
\draw[-,very thick,draw=blue] (-1,-2)--(-1,2);

% p_2 = 0:
\draw[dashed,very thick,draw=blue] (-2,-2)--(2,2);
% \draw[-,very thick,draw=red] (-2.5,-1)--(2.5,-1);

\node[circle,draw=black,fill=violet,minimum size=2mm,inner sep=0mm] at (-1,-1) {};
\node[circle,draw=black,fill=violet,minimum size=2mm,inner sep=0mm] at (1,1) {};
\end{tikzpicture}}}
\hfill
  \subfigure[Zeros of the kernel polynomial $p_1$ (lines) and $p_2$ (dashed) on $\T^2$.]{\includegraphics[width=0.3\textwidth]{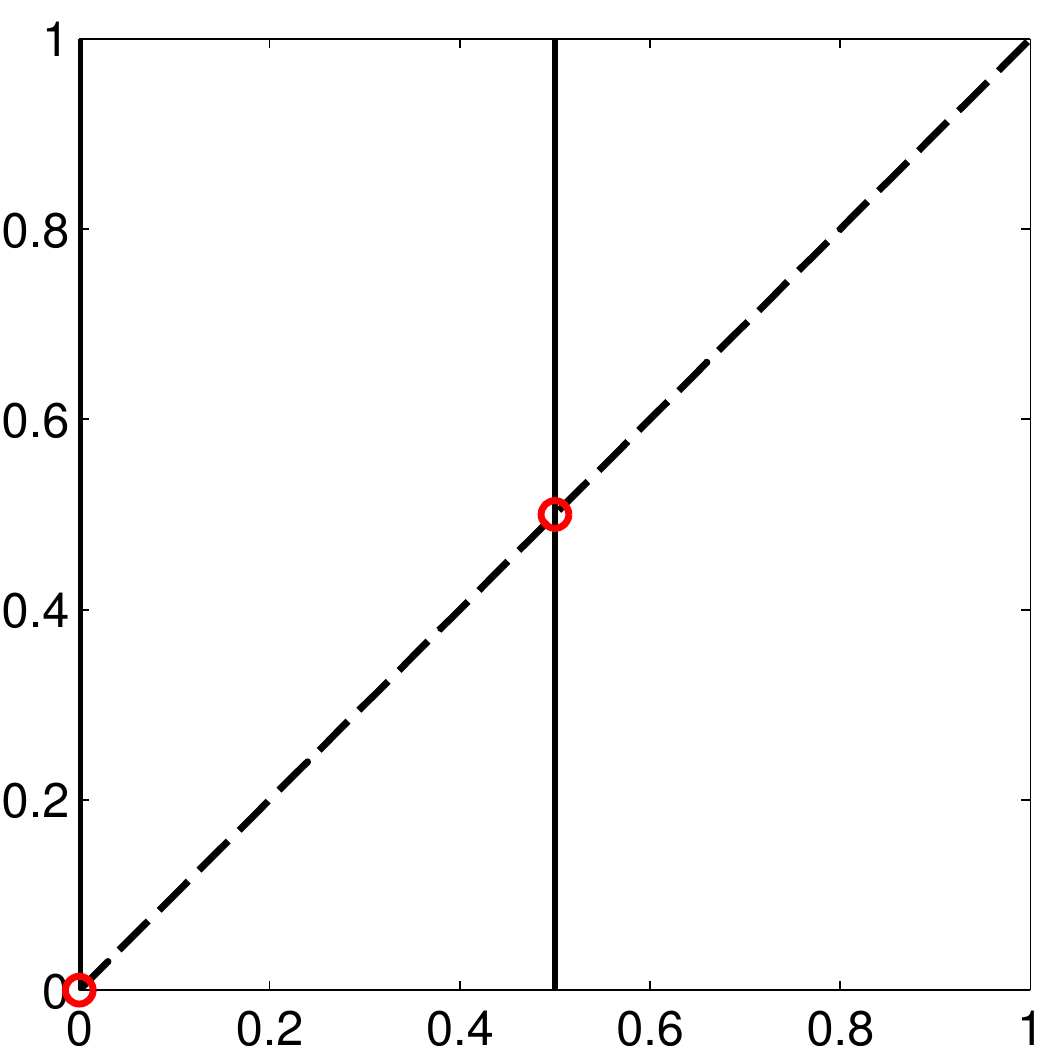}}
  \subfigure[Sum of squared absolute values of kernel polynomials.]{\includegraphics[width=0.3\textwidth]{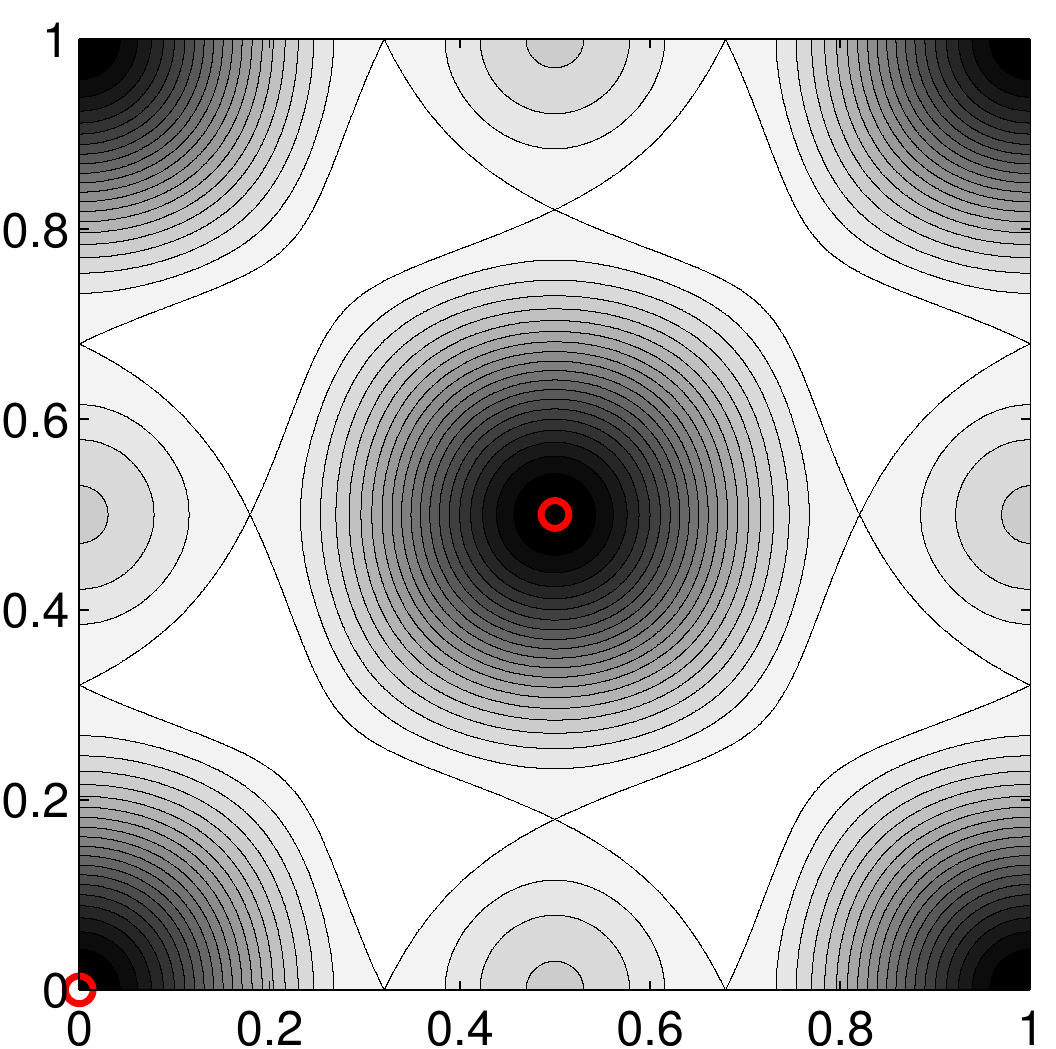}}
\caption{Parameters $d=2$, $M=2$, $n=2$,
$t_1=(0.0,0.0)$ and $t_2=(0.5,0.5)$.
}
\label{fig:toy-example-2d}
\end{figure}
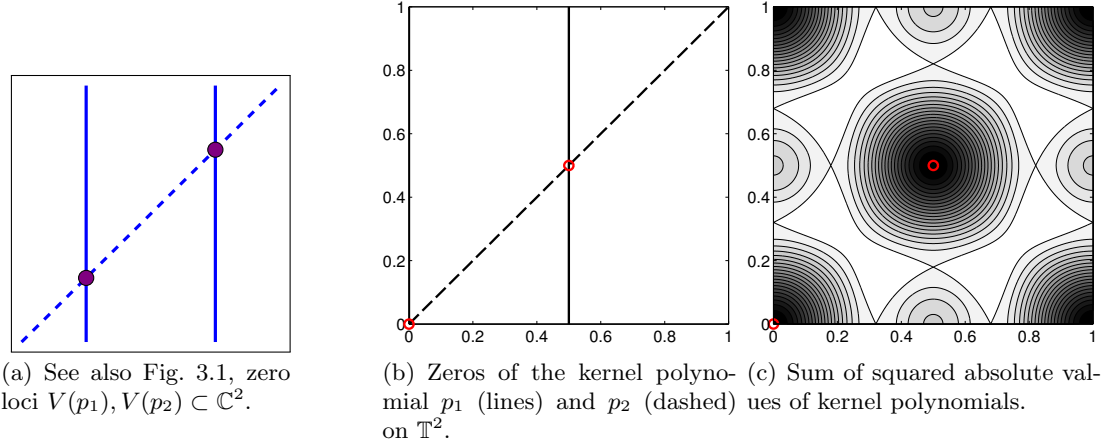

\end{example}

\begin{example}[$d=3$]
 \begin{figure}[htbp]
\centering
  \subfigure[Zeros of the kernel polynomial $p_1$ on $\T^3$.]
  {\includegraphics[width=0.45\textwidth]{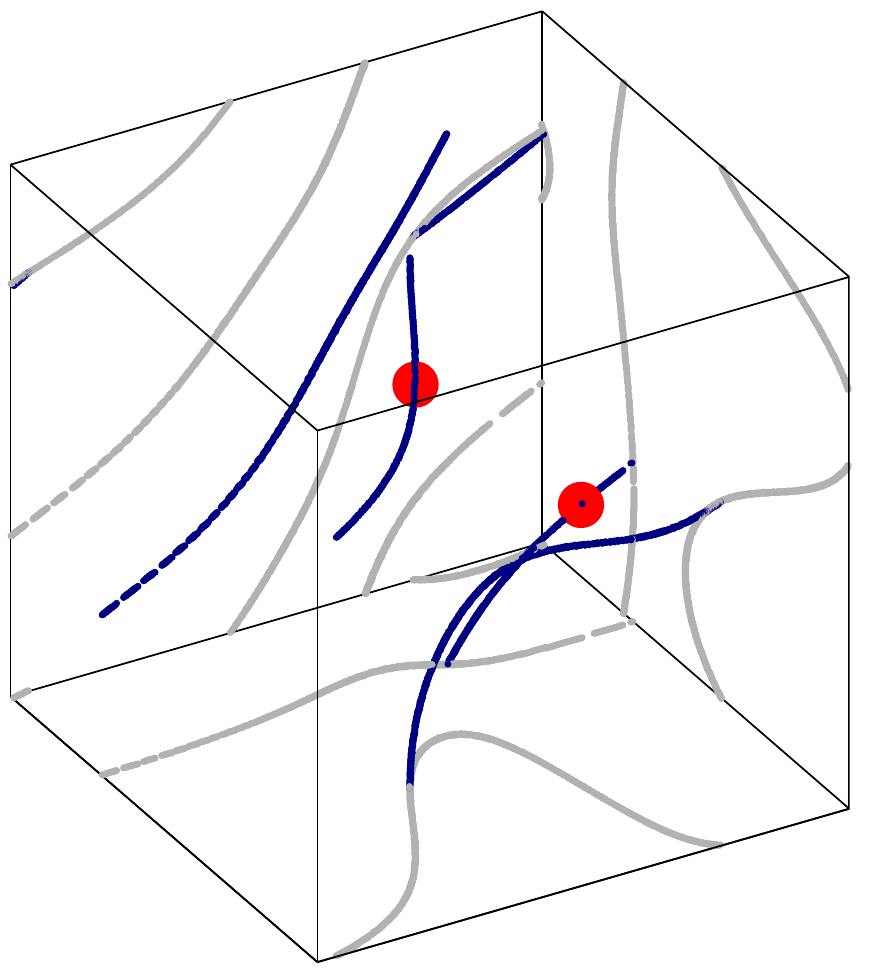}}
  \subfigure[Zeros of the kernel polynomials $p_1$, $p_2$ on $\T^3$.]
  {\includegraphics[width=0.45\textwidth]{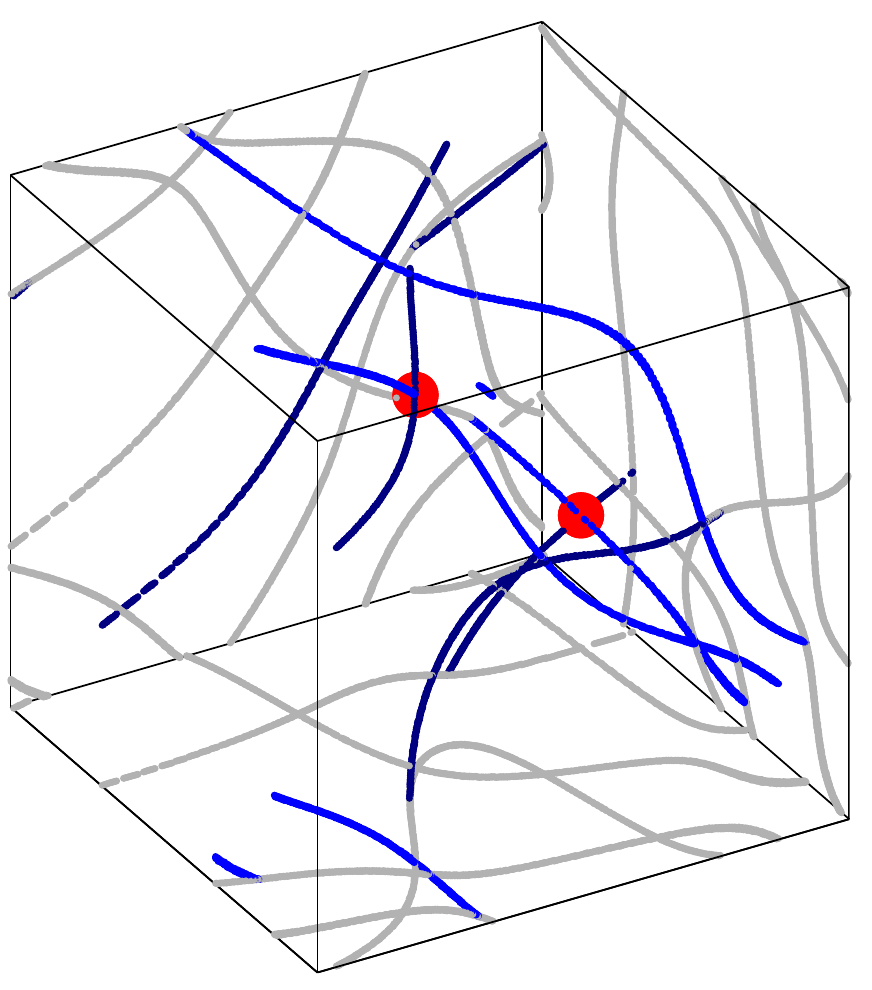}}
\caption{Parameters $d=3$, $M=2$, $n=1$, $t_1=(0.1,0.3,0.25)$ and $t_2=(0.7,0.8,0.9)$.}
\label{fig:d=3}
\end{figure}
Figure~\ref{fig:d=3}
depicts the intersection of $\T^3$
(identified with $[0,1)^3$)
and the zero loci of two polynomials
that arise with the Prony method for $M=2$ parameters
choosing $n=1$
(which is not an upper bound for $M$).
This illustrates that,
in the case $d=3$,
the zero locus of a single polynomial intersected with the torus
can typically be visualized as a ``one-dimensional'' curve
as suggested by the heuristic argument that
a complex polynomial can be thought of as two real equations,
which together with the three real equations
that define $\T^3$ as a subset of $\C^3=\R^6$
provides five equations,
thus leaving one real degree of freedom.
\end{example}

%%%%%%%%%%%%%%%%%%%%%%%%%%%%%%%%%%%%%%%%%%%%%%%%%%%%%%%%%%%%%%%%%%%%%%%%%%%%%%
\section{Summary}\label{sect:sum}

We suggested a multivariate generalization of Prony's method and gave sharp conditions
under which the problem admits a unique solution.
Moreover, we provided a tight estimate on the condition number for computing the kernel
of the involved Toeplitz matrix of moments.
Numerical examples were presented for spatial dimensions $d=1,2,3$ and showed in particular that
a so-called dual certificate in the semidefinite formulation of the moment problem can be computed
much faster by solving an eigenvalue problem.

Beyond the scope of this paper, future research needs to address
the actual computation of the common roots of the kernel polynomials,
the stable reconstruction from noisy moments, and
reductions both in the number of used moments as well as in computation time.

\textbf{Acknowledgment.}
The authors thank S.~Heider for the implementation of the approach~\cite{CaFe14} for the bivariate case and
H.~M.~M\"oller for several enlightening discussions.
The fourth author expresses his thanks to
J.~Abbott
for warm hospitality
during his visit in Genoa
and numerous useful suggestions.
% thank Daniel Potts, Christian Bey, Tim R\"omer, and Michael M\"oller...
%the referees for their valuable suggestions and
Moreover, we gratefully acknowledge support
by the DFG within the research training group 1916: Combinatorial structures in geometry
and by the Helmholtz Association within the young investigator group VH-NG-526: Fast algorithms for biomedical imaging.

%%%%%%%%%%%%%%%%%%%%%%%%%%%%%%%%%%%%%%%%%%%%%%%%%%%%%%%%%%%%%%%%%%%%%%%%%%%%%%

\bibliographystyle{abbrv}
\bibliography{../references}
\end{document}